\documentclass[a4paper,12pt]{amsart}

\usepackage[utf8]{inputenc}
\usepackage[english]{babel}

\usepackage{amsthm, amssymb}
\usepackage[alphabetic]{amsrefs}
\usepackage{enumerate}
\usepackage{color}

\usepackage[cmtip,arrow]{xy}
\usepackage{pb-diagram, pb-xy}

\usepackage{etoolbox}
\apptocmd{\sloppy}{\hbadness 10000\relax}{}{}

\newtheorem{theorem}{Theorem}[section]
\newtheorem{lemma}[theorem]{Lemma}
\newtheorem{proposition}[theorem]{Proposition}
\newtheorem{cor}[theorem]{Corollary}
\newtheorem{question}[theorem]{Question}
\theoremstyle{definition}
\newtheorem{definition}[theorem]{Definition}
\newtheorem{example}[theorem]{Example}
\theoremstyle{remark}
\newtheorem{remark}[theorem]{Remark}

\newcommand{\CC}{\mathbb{C}}
\newcommand{\NN}{\mathbb{N}}

\newcommand{\RR}{\mathbb{R}}
\newcommand{\QQ}{\mathbb{Q}}

\newcommand{\udisc}{\mathbb{D}}
\newcommand{\boundary}{\partial}
\newcommand{\aut}{\mathop{\mathrm{Aut}}}
\newcommand{\id}{\mathop{\mathrm{id}}}
\newcommand{\Olo}{\mathcal{O}}
\newcommand{\slgrp}{\mathrm{SL}}
\newcommand{\slalg}{\mathfrak{sl}}
\newcommand{\rank}{\mathop{\mathrm{rank}}}

\author{Rafael~B.~Andrist}
\address{Rafael B. Andrist \\ School of Mathematics and Natural Sciences \\ University of Wuppertal \\ Germany}
\email{rafael.andrist@math.uni-wuppertal.de}

\author{Riccardo~Ugolini}
\address{Riccardo~Ugolini \\ Institute for Mathematics, Physics and Mechanics \\ University of Ljubljana \\ Slovenia}
\email{riccardo.ugolini@imfm.si}

\title{A new notion of Tameness}
\begin{document}

\begin{abstract}
We generalize the notion of tame discrete sets introduced by Rosay and Rudin from complex-Euclidean space to arbitrary complex manifolds and establish their basic properties. We show that complex-linear algebraic groups different from the complex line or the punctured complex line contain tame discrete sets.
\end{abstract}

\maketitle

\section{Introduction}

The seminal paper \cite{RosayRudin} published in 1988 by Rosay and Rudin marked the beginning of a new era in the study of holomorphic maps and automorphisms; they introduced several new ideas that are still further developed thirty years later.

Their main area of interest was the complex Euclidean space $\CC^n$, $n>1$. In the present paper we wish to extend the concepts of tame and unavoidable discrete sets to more general complex manifolds. The standard notion of {\it tame set} is recalled at the beginning of Section $2$.

Attempts to study such sets in a setting different from the one proposed by Rosay and Rudin are already present in literature. In \cite{Kolaric2009} Kolari\v c considered automorphisms of $\CC^n$ fixing an algebraic subvariety of codimension at least $2$. Using a growth condition described by Winkelmann \cite{Winkelmann2008} he showed many interesting properties of discrete sets which are tame with respect to this subgroup of automorphisms. For a complete account on previous work we refer to \cite{FrancBook}*{Section 4.6}.

We introduce the following definition of (strongly) tame sets:
\begin{definition}
\label{def-tame}
Let $X$ be a complex manifold and let $G \subseteq \aut(X)$ be a subgroup of its group of holomorphic automorphisms. We call a closed discrete infinite set $A \subset X$ a \emph{$G$-tame set} if for every injective mapping $f \colon A \to A$ there exists a holomorphic automorphism $F \in G$ such that $F|_{A} = f$.
\end{definition}

In the next section we will discuss its relation to the standard definition given by Rosay and Rudin (see Def.~\ref{def-tame-RR}), then focus on the action of the holomorphic symplectic group on $\CC^{2n}$.

Recently, Winkelmann proposed a definition of weakly tame sets:
\begin{definition}\cite{WinkelmannTame}
Let $X$ be a complex manifold.
An infinite discrete
subset $D$ is called (weakly) tame
if for every exhaustion function $\rho \colon X \to \RR$
and every function $\zeta \colon D \to \RR$
there exists an automorphism $\Phi$ of $X$
such that $\rho(\Phi(x)) \geq \zeta(x)$ for all $x \in D$.
\end{definition}

While it is obvious that any strongly tame set is weakly tame, further investigation is required to understand when the two notions are equivalent in some suitable class of manifolds. In $\CC^n$, where they both coincide with the standard definition by Rosay and Rudin, we do have an instance of equivalence. The manifold $\udisc \times \CC$ is an example where the two notions are different \cite{WinkelmannTame}. For the rest of this paper we will use the word tame to mean strongly $\aut(X)$-tame, where $X$ will be a complex manifold deducible from context.

We provide results on Stein manifold with the density property, a class introduced by Varolin in 2001 \cite{Varolin2} for which many properties of $\CC^n$ hold.
In particular it is worth mentioning that the elements of this class have the \textit{universal embeddability property} for Stein manifolds, in the sense that all Stein manifolds can be properly embedded into any Stein manifold with the density property of sufficiently large dimension \cite{Embedded}.
Stein manifolds with the density property are central in  Anders\'en--Lempert Theory, as it is possible to approximate {\it well behaved} isotopies of injective maps from a Runge domain $\Omega \subset X$ into $X$ with isotopies of automorphisms.
For a precise statement, more details on Stein manifolds, the density property and Andersen-Lempert Theory we refer to the monograph \cite{FrancBook}*{Section 4.10}.

We also consider a symplectic version of tameness in Section \ref{sec-symplectic}.

In Section \ref{sec-sln} we turn our attention to the special linear group. This is our first example of a manifold different from the complex Euclidean space with a tame sequence. The given proofs are rather constructive, expliciting both the tame set and the interpolating automorphism.

In Section \ref{sec-lineargroups} we first establish the existence of tame sets in some large classes of manifolds with the density property, most notably Gizatullin surfaces (Corollary \ref{cor-Gizatullin}) and the Koras--Russel  cubic threefold (Corollary \ref{cor-KorasRussell}).
Then we proceed to prove the existence of tame sets in complex-linear algebraic groups (Theorem \ref{thm-lingrp}). The previous work done for $\slgrp_2(\CC)$ provides a blueprint for this discussion. 

We conclude in Section \ref{sec-unavoidable} by investigating \textit{unavoidable sets}, a class of closed discrete sets that stands at the opposite side of tame ones. We are able to generalize the methods developed in \cite{RosayRudin} to give a new proof of the existence of unavoidable sets in complex subvarieties of $\CC^n$. This result was already obtained by Winkelmann \cite{Winkelmann01} by entirely different methods. Interestingly, algebraicity enters in both proofs and statements, in our case as a growth restriction for $L^2$-estimates.

%\medskip
%When one wants to prove that a given sequence is tame finding the interpolating automorphism from Definition \ref{def-tame} can be quite a challenge, even on manifolds with the density property where the automorphism group is extremely large. To find such automorphisms we rely on complete vector fields and their kernels, which can be thought as sets of functions invariant under the flow of the corresponding vector field. The main difficulty involves searching for suitable functions in the kernel, this is deeply tied to the concept of {\it categorical quotient} in Geometric Invariant Theory and the Fourteenth Hilbert's Problem. For algebraic actions on quasi-affine manifolds the situation was clarified by Winkelmann in \cite{Winkelmann03}, his paper contains all we need for the present work. 
%To our knowledge the most recent results on this topic are contained in the work of Arzhantsev, Celik and Hausen \cite{FactorialQuotient2013}.

\section{Basic properties}
\label{sec-basic}

We begin by recalling the classical definition of tameness by Rosay and Rudin:

\begin{definition}[\cite{RosayRudin}*{Def.~3.3}]
\label{def-tame-RR}
Let $e_1$ be the first standard basis vector of $\CC^n$. A set $A \subset \CC^n$ is called \emph{tame} (in the sense of Rosay and Rudin) if there exists an $F \in \aut(\CC^n)$ such that $F(A) = \NN \cdot e_1$. It is very tame if we can choose $F \in \aut_1(\CC^n)$ to be volume preserving.
\end{definition}

We may refer to $\NN \cdot e_1 \subset \CC^n$ as the \emph{standard embedding of the natural numbers into $\CC^n$} and just write $\NN \subset \CC^n$.
The following lemma shows that these two definitions of tameness coincide for $X = \CC^n$.

\begin{lemma}
Let $A \subset \CC^n$. Then $A$ is $G$-tame with $G = \aut(\CC^n)$ in the sense of Definition \ref{def-tame} if and only if it is tame in sense of Rosay and Rudin. 
\end{lemma}

\begin{lemma} \label{eqverytame}
Let $A \subset \CC^n$. Then $A$ is $G$-tame with $G = \aut_1(\CC^n)$, the group of volume preserving automorphisms, in the sense of Definition \ref{def-tame} if and only if it is very tame in sense of Rosay and Rudin. 
\end{lemma}

\begin{proof} \hfill\newline
\framebox{$\Longrightarrow$} 
Consider all projections of the $G$-tame set $A$ to the coordinates axes in $\CC^n$. Since $A$ must be an unbounded set, at least one projection $\pi \colon \CC^n \to L$ to a certain coordinate axis must be such that $\pi(A)$ is an unbounded set. Pick an infinite discrete closed subset $B := \{b_j\}_{j \in \NN} \subset \pi(A)$. According to \cite{RosayRudin}*{Prop.~3.1, Rem.~3.4} an infinite closed discrete set $B \subset \CC^n$ that is contained in a complex line $L \subset \CC^n$ is very tame. Without loss of generality we may assume that $L$ coincides with the first coordinates axis. By the Mittag-Leffler Theorem we find a holomorphic automorphism of the form $F(z_1, z_2, \dots, z_n) = (z_1, z_2 + f_2(z_1), \dots, z_n + f_n(z_1))$ such that $F(B) \subseteq A$. By assumption of $G$-tameness for $A$ we have now found a composition of three automorphisms of $\CC^n$ that map $A$ to $\NN \subset L$, hence $A$ is tame in the sense of Rosay and Rudin. \newline
\framebox{$\Longleftarrow$} Let $A \subset \CC^n$ be a (very) tame set in the sense of Rosay and Rudin. Then every infinite subset of $B \subset A$ is (very) tame as well, again by \cite{RosayRudin}*{Prop.~3.1, Rem.~3.4}. Hence by composition of the two automorphisms furnished by the definition of (very) tameness, we obtain an (volume preserving) automorphism $F$ of $\CC^n$ with $F(A) = B$. Moreover we can achieve that for any injective given $f \colon A \to B$ the automorphism $F$ coincides with $f$ on $A$, see \cite{RosayRudin}*{Rem.~3.2}.
\end{proof}

Because of Lemma \ref{eqverytame} we will use the words \textit{very tame} when discussing $\aut_1(X)$-tame sets in a complex manifold $X$ with a holomorphic volume form.

The original Definition \ref{def-tame-RR} of Rosay and Rudin clearly implies that any two tame sequences can be mapped bijectively into one another by an automorphism of $\CC^n$, as they can both be mapped to $\NN \subset \CC^n$. It is not immediate from Definition \ref{def-tame} that any two $\aut(X)$-tame sets are equivalent, that is they can be mapped bijectively into one another by an automorphism of $X$.
We show that this is the case when $X$ is a Stein manifold with the density property.

\begin{proposition}
\label{prop-tame-equivalence}
Let $X$ be a Stein manifold with the density property and let $A,B \subset X$ be $\aut(X)$-tame sets. Then there exists $F \in \aut(X)$ such that $F(A)=B$.
\end{proposition}

\begin{proof}
Write $A=\{ a_i\}_{i \in \NN}$ and $B=\{ b_j\}_{j \in \NN}$. By the definition of $\aut(X)$-tame sets it is enough to find subsequences $\{ a_{i_k}\}_{k \in \NN} \subset A$, $\{ b_{j_k}\}_{k \in \NN} \subset B$ and $F \in \aut(X)$ such that $F(a_{i_k})=b_{j_k}$ for all $k \in \NN$.
Fix a sequence $\varepsilon_k >0$ such that $\sum_{k \geq 1} \varepsilon_k < \infty$ and an exhaustion $K_k$ of $X$ by compact sets.
We will obtain $F$ as the limit of a composition of automorphisms $F_k=\phi_k \circ \phi_{k-1} \circ \dots \circ \phi_1$ with the following properties:

\begin{itemize}
\item[\rm (i)] $F_k(a_{i_r})=b_{j_r}$ for all $r \leq k$;
\vspace{1mm} 
\item[\rm (ii)] $\phi_k$ is $\varepsilon_k$-close to the identity on $C_{k-1}$,

where $C_k$ is an exhaustion of $X$ by holomorphically convex compact sets such that

\item[\rm (iii)] $K_k \cup F_k(K_k) \subset C_k$;
\vspace{1mm} 
\item[\rm (iv)] $C_{k-1} \subset \mathring{C_k}$ and $dist(X \setminus C_k, C_{k-1})>\varepsilon_k$.
\end{itemize}

Suppose we have such a sequence. Then it converges (uniformly on compacts) to $F \in \aut(X)$ thanks to \cite{Forstneric1999}*{Proposition 5.1} and it has the required property because of (i).

We will construct the automorphisms $\phi_k$ by induction. For $k=1$, let $j_k = 1$ and $C_0=\emptyset$. Then $\phi_1$ is any automorphism such that $\phi_1(a_1)=b_1$ which exists because $X$ has the density property.
Suppose we have chosen $j_r$ and constructed $\phi_r$ and $C_{r-1}$ as above for all $r \leq k$. First of all pick a holomorphically convex compact set $C_k$ satisfying (iii) and (iv) above and such that $b_{j_r} \in C_k$ for all $r \leq k$, this is possible because $X$ is Stein. Choose $j_{k+1}$ such that $a_{j_{k+1}}, b_{j_{k+1}} \in X \setminus C_k$, this can be done as $A$ and $B$ are necessarily unbounded. By means of the Anders{\'e}n--Lempert theory, see \cite{PresentState}*{Prop.~2.1}, we can then obtain an automorphism $\phi_{k+1}$ which is $\varepsilon_{k+1}$-close to the identity on $C_k$ and is such that $\phi_{k+1}(b_{j_r})=b_{j_r}$ for $r \leq k$ and $\phi_{k+1}(a_{j_{k+1}})= b_{j_{k+1}}$.
The composition $F_{k+1}=\phi_{k+1} \circ \phi_k \circ \dots \circ \phi_1$ then satisfies (i) and we can proceed with the induction.
\end{proof}

\section{Symplectic case}
\label{sec-symplectic}

We are also interested in a special class of volume preserving automorphisms, namely the ones preserving the standard symplectic form $\omega_{Sp}=\sum_{j=1}^{n} dz_j \wedge dz_{n+j}$ on $\CC^{2n}$.

\begin{definition}
We denote by $\aut_{Sp}(\CC^{2n})$ the group of automorphisms of $\CC^{2n}$ that preserve $\omega_{Sp}$ and call it the \emph{group of symplectic automorphisms}.
\end{definition}

It was shown by Forstneri\v{c} \cite{ForstnericActions}*{Thm.~5.1} that the group of symplectic automorphisms $\aut_{Sp}(\CC^{2n})$ contains a dense subgroup generated by shears of the following form
\[
F(z)=z + f(z^TJv)v,
\]
where $v \in \CC^{2n}$, $f\colon \CC \to \CC$ is an entire function, $z^T$ denotes the transpose of $z$ and $J$ is the $2n \times 2n$ block matrix
\[\begin{bmatrix}
   0   & \id  \\
    -\id     & 0
\end{bmatrix}.
\]

While Rosay and Rudin were not interested in the study of tameness with respect to this group, we have an obvious extension of the definition of tame and very tame sets.

\begin{definition}
\label{def-symp-tame-RR}
A set $A \subset \CC^{2n}$ is called \emph{symplectically tame} (in the sense of Rosay and Rudin) if there exists an $F \in \aut_{Sp}(\CC^{2n})$ such that $F(A) = \NN \cdot e_1$.
\end{definition}

As with tame and very tame, we wish to verify that we can permute any two points of $\NN \cdot e_1$ with an element of $\aut_{Sp}(\CC^{2n})$. We will show that any two such sequences contained in the $z_1$-axis of $\CC^{2n}$ can be mapped into one another by an element of $\aut_{Sp}(\CC^{2n})$. The following is analogous to \cite{RosayRudin}*{Prop.~3.1}.

\begin{lemma}
If $\{\alpha_j\}$ and $\{\beta_j\}$ are closed discrete sequences in $\CC$, $\alpha_i \neq \alpha_j$ and $\beta_i \neq \beta_j$ for $i \neq j$, then there exists $F \in \aut_{Sp}(\CC^{2n})$ such that $F(\alpha_j e_1)=\beta_j e_1$ for $j \in \NN$.
\end{lemma}

\begin{proof}
Let $\sigma_1(z)=z+z_1 e_{n+1}=z+(z^TJe_{n+1}) e_{n+1}$. By the Mittag-Leffler interpolation theorem there are $f,g \colon \CC \to \CC$ such that $f(\alpha_j)=\beta_j-\alpha_j$ and $g(\beta_j)=-\alpha_j$. Choose $\sigma_2(z)= z+f(z_{n+1}) e_1 =z+f(-z^T J e_1) e_1$ and $\sigma_3(z)= z + g(z_1) e_{n+1}  = z+g(z^T J e_{n+1}) e_{n+1}$. Then $F=\sigma_3 \circ \sigma_2 \circ \sigma_1$ is the required symplectic automorphism.
\end{proof}

We will now compare Definition \ref{def-symp-tame-RR} with Definition \ref{def-tame} for $G=\aut_{Sp}(\CC^{2n})$.

\begin{lemma}
Let $A \subset \CC^{2n}$. Then $A$ is $G$-tame with $G = \aut_{Sp}(\CC^{2n})$ in the sense of Definition \ref{def-tame} if and only if it is symplectically tame in sense of Rosay and Rudin. 
\end{lemma}

\begin{proof}
The proof is similar to the one presented for tame and very tame sets. As we have already enstablished that a discrete subset $B \subset \pi(A)$ of a complex line is symplectically tame, we need to show that we can construct the automorphism $F$ in the \framebox{$\Longrightarrow$} direction of the above proof in such a way that $F \in \aut_{Sp}(\CC^{2n})$. \newline
Again, suppose $\pi=\pi_1$ is the projection on the first coordinate axis. We wish to find $F \in \aut_{Sp}(\CC^{2n})$ such that $F(B) \subset A$. We will obtain it as a composition $\sigma_{2n}\circ \dots \circ \sigma_2$. Enumerate $B=\{b_i\}_{i \in \NN}$ and for each $b_i \in B$ choose $a_i \in A \cap \pi^{-1}(b_i)$. Denote by $\pi_j \colon \CC^{2n} \to \CC$ the projection onto the $j$-th coordinate.\newline
For $2 \leq j\leq n$, let $f_j \colon \CC \to \CC$ be such that $f_j(\pi_1 (b_i))=\pi_j(a_i)$ for $i \in \NN$ and define $\sigma_j(z)=z+f_j(z^T J (e_j+e_{n+1}))(e_j+e_{n+1})$. Observe that $\pi_j(\sigma_{n}\circ \dots \circ \sigma_2(b_i))=\pi_j(a_i)$ for $1 \leq j \leq n$ and $i \in \NN$.\newline
Let $c_i=\sigma_{n}\circ \dots \circ \sigma_2(b_i)$ and for $2 \leq j \leq n-1$ pick functions $f_{j+n} \colon \CC \to \CC$ such that $f_{j+n}(\pi_j(c_i)+\pi_1(c_i))=\pi_{n+j}(a_i)$ for all $i \in \NN$. If we define $\sigma_{n+j}(z)=z+f_{n+j}(z^T J (e_{n+j}+e_{n+1}))(e_{n+j}+e_{n+1})$, then $\pi_j(\sigma_{2n-1}\circ \dots \circ \sigma_2(b_i))=\pi_j(a_i)$ for $1 \leq j \leq 2n-1, \ j \neq n+1$ and $i \in \NN$. Hence all projections are fine but the one on the $(n+1)$-th coordinate axis. Denote by $d_i=\sigma_{2n-1}\circ \dots \circ \sigma_2(b_i)$ and let $f_{2n} \colon \CC \to \CC$ be such that $f_{2n}(\pi_1(b_i))=\pi_{n+1}(a_i)-\pi_{n+1}(d_i)$. For $\sigma_{2n}(z)=z+f_{2n}(z^T J e_{n+1}) e_{n+1}$, we obtain $F=\sigma_{2n}\circ \dots \circ \sigma_2$ with the required properties.
\end{proof}

\begin{lemma}
The set $\NN \times \NN \times \{0\} \times \{0\} \subset \CC^4$ is $\aut_{Sp}(\CC^4)$-tame.
\end{lemma}

\begin{proof}
Let $\alpha \in \RR \setminus \QQ$ be a positive irrational number. Let $\sigma_1 \in \aut_{Sp}(\CC^4)$ of the form
\[
\sigma_1(z)=z+f(z_1+\alpha z_3)(e_2 + \alpha e_4),
\]
where $f \in \Olo(\CC)$ is such that $f(n+\alpha m)=x_{n,m}$ and $\{x_{n,m}\}_{n,m \in \NN}$ is any closed discrete set in $\CC$, for instance the enumeration of $\NN \times \NN$ given by Cantor diagonal argument.
Let $\sigma_2 \in \aut_{Sp}(\CC^4)$ satisfy
\[
\sigma_2(z)=z+g(z_3)e_1 +h(z_4)e_2,
\]
where $g,h \in \Olo(\CC)$ such that $g(x_{n,m})=x_{n,m}-n$ and $h(\alpha x_{n,m})=-m$.
Finally we have $\sigma_3 \in \aut_{Sp}(\CC^4)$ of the form
\[
\sigma_3(z)=z+b(z_1+ \alpha z_2)(e_3+\alpha e_4)
\]
for $b\in \Olo(\CC)$ such that $b(x_{n,m})=-x_{n,m}$.

The composition $\sigma_3 \circ \sigma_2 \circ \sigma_1 \in \aut_{Sp}(\CC^4)$ then maps $(n,m,0,0)$ to $(x_{n,m},0,0,0)$, proving the statement.
\end{proof}

\begin{proposition}
Let $A \subset \CC^4$ be a closed discrete set and $\pi_1,\pi_2: \CC^4 \to \CC$ the projections on the first two coordinate axes respectively. If $\pi_1(A)$ and $\pi_2(A)$ are closed discrete and with finite fibers, then $A$ is $\aut_{Sp}(\CC^4)$-tame.
\end{proposition}

\begin{proof}
We will map $A$ to a subset of $\NN \times \NN \times \{0\} \times \{0\}$ using symplectic automorphisms. Observe that if $F=(F_1,F_2) \in \aut_1(\CC^2)$, then $(F_1(z_1,z_3),z_2,F_2(z_1,z_3),z_4) \in \aut_{Sp}(\CC^4)$.
Consider the projections
\[
\begin{diagram}
\node{\CC^4} \arrow{e,t}{\pi_{1,3}} \node{\CC^2} \arrow{e,t}{\tilde{\pi}_1} \node{\CC} \\
\node{(z_1,z_2,z_3,z_4)} \arrow{e,t,T}{} \node{(z_1,z_3)} \arrow{e,t,T}{} \node{z_1} \\
\end{diagram}
\]
As $\pi_1(A)$ is closed discrete and with finite fibers so is $\tilde{\pi}_1$ when restricted to $\pi_{1,3}(A)$, hence  $\pi_{1,3}(A)$ is very tame in $\CC^2$ thanks to \cite{RosayRudin}*{Theorem 3.5}.
Therefore there exists $F \in \aut_1(\CC^2)$ such that $F(\pi_{1,3}(A)) \subset \NN \times \{0\}$. By the previous observation we also have a symplectic automorphism such that $A$ is mapped into $\NN \times \CC \times \{0\} \times \CC$.
By the same argument applied to the second and fourth coordinate, we have mapped $A$ in $\NN \times \NN \times \{0\} \times \{0\}$.
\end{proof}

\section{Special Linear Group}
\label{sec-sln}

In this section we will provide a few examples of tame sequences in the special linear group $\slgrp_2(\CC)$.

\begin{lemma} \label{seq1}
The set
\[
\left\{ \begin{pmatrix}
1 & k \\
0 & 1
\end{pmatrix} \,:\, k \in \NN \right\} \subset \slgrp_2(\CC)
\]
is $G$-tame for $G = \aut(\slgrp_2(\CC))$
\end{lemma}

Before we start with the actual proof of the lemma, let us recall some facts about the special linear group $\slgrp_2(\CC)$ and its Lie algebra $\slalg_2(\CC)$ in the adjoint representation.

In the following we will fix the coordinates for
\[
\slgrp_2(\CC) = \left\{ \begin{pmatrix} a & b \\ c & d \end{pmatrix} \,:\, ad - bc = 1 \right\}
.\]

There are three $\CC$-complete vector fields that generate $\slalg_2(\CC)$ and correspond to conjugations:
\begin{align*}
V &= c \frac{\partial}{\partial a} + (d-a) \frac{\partial}{\partial b} - c  \frac{\partial}{\partial d} \\
W &= -b \frac{\partial}{\partial a} + (a-d) \frac{\partial}{\partial c} + b \frac{\partial}{\partial d} \\
H &= -2b \frac{\partial}{\partial b} + 2c \frac{\partial}{\partial c}
\end{align*}
satisfying $[V, W] = H$, $[H,V] = 2V$ and $[H, W] = 2W$,
with their respective flows
\begin{align*}
\varphi_V^t &= \begin{pmatrix} 1 & t \\ 0 & 1 \end{pmatrix} \cdot \begin{pmatrix} a & b \\ c & d \end{pmatrix} \cdot\begin{pmatrix} 1 & -t \\ 0 & 1 \end{pmatrix} = \begin{pmatrix}c t+a & t \left( d-c t\right) -a t+b\\ c & d-c t\end{pmatrix} \\
\varphi_W^t &= \begin{pmatrix} 1 & 0 \\ t & 1 \end{pmatrix} \cdot \begin{pmatrix} a & b \\ c & d \end{pmatrix} \cdot \begin{pmatrix} 1 & 0 \\ -t & 1 \end{pmatrix} = \begin{pmatrix}a-b t & b \\ t \left( a-b t\right) -d t+c & b t+d\end{pmatrix}  \\
\varphi_H^t &= \begin{pmatrix} e^{-t} & 0 \\ 0 & e^{t} \end{pmatrix} \cdot \begin{pmatrix} a & b \\ c & d \end{pmatrix} \cdot \begin{pmatrix} e^{t} & 0 \\ 0 & e^{-t} \end{pmatrix}
\end{align*}

We also have the following relations for the kernels:
\begin{align*}
\ker {V} &\supseteq \CC[c,a+d] \\
\ker {W} &\supseteq \CC[b,a+d] 
\end{align*}
And for the second kernels we have:
\begin{align*}
\ker {V^2} &\supseteq \CC[a,c,d] \\
\ker {W^2} &\supseteq \CC[a,b,d] 
\end{align*}
Note that $ad - bc = 1$ is trivially contained in all the kernels.

\begin{proof}
We are given an injective map $\ell \colon \NN \to \NN$ that prescribes the injection
\[
\begin{pmatrix}
1 & k \\
0 & 1
\end{pmatrix}
\mapsto
\begin{pmatrix}
1 & \ell(k) \\
0 & 1
\end{pmatrix}
.\]

Next, we construct an interpolating holomorphic automorphism of $\slgrp_2(\CC)$ as a composition of suitable time-$1$ maps of complete vector fields corresponding to conjugations.

The desired automorphism is given by:
\[
\varphi_W^{G} \circ \varphi_V^{F}  \circ \varphi_W^1 
\]
where $F$ and $G$ are holomorphic functions given as follows:
\begin{enumerate}
\item \[ F\left( \begin{pmatrix} a & b \\ c & d \end{pmatrix} \right) = f(-c) \]
By the Mittag-Leffler interpolation theorem we find a holomorphic function $f \colon \CC \to \CC$ such that $f(k) = \sqrt{\frac{\ell(k)}{k}} - 1$ for all $k \in \NN$. The root can be chosen arbitrary. Note that $F \in \ker V$, hence $F \cdot V$ is $\CC$-complete and $\varphi_V^{F}$ is a holomorphic automorphism.
\item \[ G\left( \begin{pmatrix} a & b \\ c & d \end{pmatrix} \right) = g(b) \]
Note that $\ell \colon \NN \to \NN$ is injective and has in particular a closed discrete image. Again by the Mittag-Leffler interpolation theorem we find a holomorphic function $g \colon \CC \to \CC$ such that $g(\ell^{-1}(k)) = - \frac{1}{1 + f(k)} - 1$ for all $k \in \NN$. This is well defined, since $f(k) \neq - 1$ for all $k \in \NN$ by construction. Note that $G \in \ker W$, hence $G \cdot W$ is $\CC$-complete and $\varphi_V^{F}$ is a holomorphic automorphism.
\end{enumerate}
We now check that we indeed interpolate correctly:
\begin{align*}
\begin{pmatrix} 1 & k \\ 0 & 1 \end{pmatrix}
&\stackrel{\varphi_W^1}{\mapsto}
\begin{pmatrix} 1-k & k \\ -k & k+1\end{pmatrix} \\
&\stackrel{\varphi_V^F}{\mapsto}
\begin{pmatrix} -k f(k) -k+1 & k \cdot \left( 1+f(k) \right)^2 \\ -k & k \cdot f(k) + k+1 \end{pmatrix} \\
&\stackrel{\varphi_W^G}{\mapsto} 
\begin{pmatrix} 1 & k\cdot \left( 1+ f(k)\right)^2 \\ 0 & 1\end{pmatrix}
=
\begin{pmatrix} 1 & \ell(k) \\ 0 & 1 \end{pmatrix} \qedhere
\end{align*}

\end{proof}

\begin{remark}
The $\aut(\slgrp_2(\CC))$-tame set above is easily seen to be also $\aut_\omega(\slgrp_2(\CC))$-tame where $\omega$ is the bi-invariant Haar form on $\slgrp_2(\CC)$, since we use only shears of $\slgrp_2(\CC)$-actions by left- or right-multiplication, and shears do not change the vanishing $\omega$-divergence.
\end{remark}

\begin{lemma} \label{seq2}
The set
\[
\left\{ \begin{pmatrix}
k & 0 \\
0 & \frac{1}{k}
\end{pmatrix} \,:\, k \in \NN \right\} \subset \slgrp_2(\CC)
\]
is $G$-tame for $G = \aut(\slgrp_2(\CC))$.
\end{lemma}

To prove this fact we will use a combination of left and right multiplication. Consider the complete vector fields

\begin{align*}
A &= c \frac{\partial}{\partial a} + d \frac{\partial}{\partial b} \\
B &= a \frac{\partial}{\partial c} +b \frac{\partial}{\partial d} \\
C &= a \frac{\partial}{\partial b} +c \frac{\partial}{\partial d}
\end{align*}

and the respective flows

\begin{align*}
\varphi_A^t &= \begin{pmatrix} 1 & t \\ 0 & 1 \end{pmatrix} \cdot \begin{pmatrix} a & b \\ c & d \end{pmatrix} \\
\varphi_B^t &= \begin{pmatrix} 1 & 0 \\ t & 1 \end{pmatrix} \cdot \begin{pmatrix} a & b \\ c & d \end{pmatrix} \\
\varphi_C^t &= \begin{pmatrix} a & b \\ c & d \end{pmatrix} \cdot \begin{pmatrix} 1 & t \\ 0 & 1 \end{pmatrix}
\end{align*}

In this case we have the following:

\begin{align*}
\ker {A} &\supseteq \CC[c,d] \\
\ker {B} &\supseteq \CC[a,b] \\
\ker {C} &\supseteq \CC[a,c] 
\end{align*}

\begin{proof}
Denote the given injective map by
\[
\begin{pmatrix}
k & 0 \\
0 & \frac{1}{k}
\end{pmatrix}
\mapsto
\begin{pmatrix}
\ell(k) & 0 \\
0 & \frac{1}{\ell(k)}
\end{pmatrix}
.\]

As before, we will obtain the interpolating automorphism as a composition of time-1 flows of suitably defined vector fields. Choose holomorphic functions $f=f(c) \in \ker A$, $g = g(a) \in \ker B$ and $h =h(a)\in \ker C$ such that
\begin{align*}
f\begin{pmatrix}
k & 0 \\
k & \frac{1}{k}
\end{pmatrix}&=\frac{\ell(k)}{k}-1 \\
g\begin{pmatrix}
\ell(k) & \frac{\ell(k)}{k^2}-\frac{1}{k} \\
k & \frac{1}{k}
\end{pmatrix}&=-\frac{k}{\ell(k)} \\
h\begin{pmatrix}
\ell(k) & \frac{\ell(k)}{k^2}-\frac{1}{k} \\
0 & \frac{1}{\ell(k)}
\end{pmatrix}&=\frac{1}{k \ell(k)}-\frac{1}{k^2}.
\end{align*}

Let $F,G,H \in \aut(\slgrp_2(\CC))$ be the time-1 flows of $fA, gB$ and $hC$ respectively. Then
\[
H\circ G\circ F \left(\begin{pmatrix}
1 & 0 \\
1 & 1
\end{pmatrix}
\begin{pmatrix}
k & 0 \\
0 & \frac{1}{k}
\end{pmatrix}
\right)=\begin{pmatrix}
\ell(k) & 0 \\
0 & \frac{1}{\ell(k)}
\end{pmatrix}.
\]
\end{proof}

%\smallskip
%Notable special cases:
%
%Up to a translation, the natural action of $\slgrp_2(\CC)$ on $\CC^2$ reproduces the standard tame set:
%\[
%\begin{pmatrix}
%1 & k \\ 0 & 1
%\end{pmatrix} \cdot
%\begin{pmatrix}
%0 \\ 1
%\end{pmatrix} =
%\begin{pmatrix}
%k \\ 1
%\end{pmatrix}
%\]
%
%For $\CC^\ast \times \CC^\ast$, we can basically repeat the procedure for tame sets in $\CC^2$, but with multiplicative instead of additive shears. The set $A := \{ (e^k,1) \,:\, k \in \NN \} \subset \CC^\ast \times \CC^\ast$ is $\aut(\CC^\ast \times \CC^\ast)$-tame:
%Let $\varphi \colon \NN \to \NN$ describe an injective self-map of $A$ by $(e^k,1) \mapsto (e^{\varphi(k)},1)$. Let $\Phi \colon \CC^\ast \to \CC^\ast$ be a holomorphic map such that $\Phi(e^k) = e^{\varphi(k) - k}$. Let $\Psi \colon \CC^\ast \to \CC^\ast$ by a holomorphic map such that $\Psi(e^{\varphi(k)}) = e^{-k}$.
%We now give the interpolating automorphism as composition of three multiplicative shears:
%\begin{align*}
%F(z,w) &= (z, z \cdot w)\\
%G(z,w) &= (\Phi(w) \cdot z, w) \\
%H(z,w) &= (z, \Psi(z) \cdot w)
%\end{align*}
%The desired automorphism is $H \circ G \circ F$.

\begin{cor}
There exists $F \in \aut(\slgrp_2(\CC))$ such that 
\[
F \begin{pmatrix}
k & 0 \\
0 & \frac{1}{k}
\end{pmatrix} =  \begin{pmatrix}
1 & k \\
0 & 1
\end{pmatrix}
\]
for all $k \in \NN$.
\end{cor}

\begin{proof}
Lemmas \ref{seq1} and \ref{seq2} give that the two sequences are tame. Since $\slgrp_2(\CC)$ has the density property \cite{TothVarolin2000}, the result is a consequence of Proposition \ref{prop-tame-equivalence}.
\end{proof}

\section{Existence results}
\label{sec-lineargroups}

\begin{lemma}
\label{lem-product}
Let $X$ and $Y$ be Stein manifolds and let $V$ resp.\ $W$ be a complete holomorphic vector field on $X$ resp.\ $Y$ with at least one non-periodic unbounded orbit. Then there exists an $\aut(X \times Y)$-tame set in $X \times Y$.
\end{lemma}
\begin{proof}
Let $\phi_V^t$ resp.\ $\phi_W^t$ denote the flow map of the non-zero complete holomorphic vector field $V$ resp.\ $W$ on $X$ resp.\ $Y$.

%By $\pi_X \colon X \times Y \to X$ resp.\ $\pi_Y \colon X \times Y \to Y$ we denote the canonical projections.

Choose $x_0 \in X$ resp.\ $y_0 \in Y$ such that its orbit under $\phi_V^t$ resp.\ $\phi_W^t$ is non-periodic and unbounded. We define the tame set $A$ to be $A := \{ \phi_V^n(x_0) \,:\, n \in \NN \} \times \{ y_0 \} \subset X \times Y$. By assumption, $A$ is closed and discrete.

We need to show that for any injective map $\ell \colon \NN \to \NN$ we find a holomorphic automorphism $F$ of $X \times Y$ such that $F( \phi_V^n(x_0), y_0 ) = ( \phi_V^{\ell(n)}(x_0), y_0 )$.

We construct this automorphism as a composition $F = F_3 \circ F_2 \circ F_1$ where
\begin{align*}
F_1(x, y) &= (x, \phi_W^x(y)) \\
F_2(x, y) &= (\phi_V^{f(y)}(x), y)) \\
F_3(x, y) &= (x, \phi_W^{g(x)}(y))
\end{align*}
where we choose holomorphic functions $f \colon Y \to \CC$ and $g \colon X \to \CC$ as follows:
\begin{align*}
f( \phi_W^n(y_0) ) &= \ell(n) - n \\
g( \phi_V^{\ell(n)}(x_0) ) &= -n
\end{align*}
Such functions exist, since we can prescribe the values of holomorphic function on a closed discrete set in a Stein manifold.
\end{proof}

\begin{cor}
For $m, n \in \NN$ with $m + n \geq 2$ the manifold $\CC^m \times (\CC^\ast)^n$ contains a tame set.
\end{cor}
\begin{proof}
Note that $z \frac{\partial}{\partial z}$ is a complete holomorphic vector field on $\CC^\ast$ resp.\ $\frac{\partial}{\partial z}$ is a complete holomorphic vector field on $\CC$ such that the assumptions of the preceding lemma are satisfied.
\end{proof}

Due to an important application for a large class of affine surfaces, we give another variation of this statement and proof.
\begin{lemma}
Let $X=\CC_z \times \CC^\ast_w$. Consider the complete flow maps 
$\varphi^t(z,w) = (z + w^m t, w)$ of $w^m \frac{\partial}{\partial z}$ and 
$\psi^t(z,w) = (z, \exp(z^m t) w)$ of $z^m w \frac{\partial}{\partial w}$.
Then the tameness of the set $A = \left\{ (1, n) \,:\, n \in \NN \right\}$ can be established using only the flows $\varphi^t$ and $\psi^t$.
\end{lemma}

\begin{proof}
Let $\ell \colon \NN \to \NN$ be any injective map. To establish tameness, we need to show that there exists a holomorphic automorphism mapping $(1, n)$ to $(1, \ell(n))$.

We construct this automorphism as a composition $F = F_3 \circ F_2 \circ F_1$ and define those three automorphisms as follows:
\begin{align*}
F_1(z, w) &= \psi^z(z, w) \\
F_2(z, w) &= \varphi^{g(w)}(z, w) \\
F_3(z, w) &= \psi^{h(z)}(z, w)
\end{align*}
where we choose the holomorphic functions $g$ and $h$ to have the following prescribed values:
\begin{align*}
g(\exp(n^{m+1})) &= \exp( -m n^{m+1}) \cdot (\ell(n) - n) \\
h(\ell(n)) &= -(\ell(n))^{-m} \cdot n^{m+1}
\end{align*}
Tracing the images of $(1, n)$ under the composition $F = F_3 \circ F_2 \circ F_1$, it is then easy to verify that $F$ is the desired automorphism.
\end{proof}

\begin{cor}
\label{cor-Gizatullin}
Every Gizatullin surface $X$ with a reduced singular fibre contains an $\aut(X)$-tame set.
\end{cor}
\begin{proof}
See \cite{AndristSurfaces} for an affine chart $\CC_z \times \CC^\ast_w$ of $X$ which is such that the vector fields $w^m \frac{\partial}{\partial z}$ and $z^m w \frac{\partial}{\partial w}$ extend as complete holomorphic vector fields to $X$ provided $m$ is chosen large enough, depending on $X$.
\end{proof}

\begin{example}
This includes for example the Danielewski surfaces which are given by
\[
\{ (x,y,z) \in \CC^3 \,:\, x^k y = p(z) \}
\]
where $p$ is a polynomial with simple roots.
A suitable affine coordinate chart is given by
\[
\CC^\ast \times \CC \ni (x, z) \mapsto (x, p(z)/x^k ,z)
\]
\end{example}

We already have many examples of manifolds admitting tame sets, yet we are still missing a somewhat general theorem concerning a large class of complex manifolds. What follows is some standard theory of invariant functions which we will use to discuss the case of Lie groups.

Given an affine algebraic manifold $X$ and a complete algebraic vector field $V$ whose flow $\phi_V^t$ is algebraic, we have an algebraic action of the additive group $(\CC, +)$:
\[
\begin{split}
\CC \times X& \to X \\
(t,x)& \mapsto \phi_V^t(x)
\end{split}
\]
This gives an action on the algebra of regular functions $\CC[X]$ on $X$. We observe that the functions invariant under this action are precisely the ones in the kernel of $V$. Therefore we are interested in the ring of invariant functions $ \CC[X]^{(\CC,+)}$ with respect to the action of $\phi_V^t$. If $ \CC[X]^{(\CC,+)}$ is finitely generated, its spectrum is the GIT quotient considered in geometric invariant theory and it is affine. When this ring is not finitely generated, we will use the following result of Winkelmann as a substitute:

\begin{theorem}[\cite{Winkelmann03}*{Theorem~3}] \label{Wink}
Let $k$ be a field, $\Omega$ an irreducible, reduced, normal $k$-variety and $G \subset \aut(V)$.

Then there exists a quasi-affine $k$-variety $Z$ and a rational map $\pi \colon \Omega \to Z$ such that
\begin{enumerate}
\item The rational map $\pi$ induces an inclusion $\pi^{*} \colon k[Z]\subset k[\Omega]$.
\item The image of the pull-back $\pi^*(k[Z])$ coincides with the ring of invariant functions $k[\Omega]^G$.
\item For every affine $k$-variety $\Theta$ and every $G$-invariant morphism $f \colon \Omega \to \Theta$ there exists a morphism $F \colon Z \to \Theta$ such that $F \circ \pi$ is a morphism and $f = F \circ \pi$.
\end{enumerate}
\end{theorem}

%\begin{remark}
%The inclusion $\pi^{*} \colon \Olo(Z) \to \Olo(\Omega)$ holds also for the holomorphic functions\footnote{Let $Z = Z^\prime \setminus A$ for an affine variety $Z^\prime$ and a Zariski-closed subset $A$. Without loss of generality we may assume that $A$ has codimension at least $2$, otherwise $Z$ would be affine. We may pass to the normalization $\tilde Z$ of $Z^\prime$. Every holomorphic function on $Z^\prime \setminus A$ defines a holomorphic function on the normalization $\tilde Z \setminus \tilde A$ and extends to $\tilde Z$ where it can be approximated by regular functions. Since $\tilde Z$ is the normalization of $Z^\prime$, the rational map $\pi \colon \Omega \to Z$ can be lifted to a rational $\tilde \pi \colon \Omega \to \tilde Z$. Then $\tilde \pi^{*}$ maps holomorphic functions on $\tilde Z$ to a priori meromorphic functions on $\Omega$.}.
%\end{remark}

When $G=(\CC,+)$ and the action is given by a complete vector field $V$, we will denote the quotient by $\pi_V \colon X \to Q_V$.
This is a useful construction to produce functions in the kernel of such a derivation and we will use it repeatedly in the upcoming discussion.

A crucial ingredient will also be actions of $\slgrp_2(\CC)$ which is a reductive group and therefore the ring of invariant functions is finitely generated, see e.g.\ the textbook of Freudenburg \cite{LND}*{Section~6.1}.

\begin{proposition} \label{prop-simplegroup}
Let $G$ be a linear algebraic group and assume there exists an algebraic Lie group homomorphism $i \colon \slgrp_2(\CC) \to G$ that is immersive. Then $G$ contains a $\aut(G)$-tame set.
\end{proposition}

\begin{proof}
Let $V$ and $W$ be the vector fields on $G$ with flows given by
\[
\begin{split}
\CC \times G& \to G \\
(t,x)& \to i \begin{pmatrix} 1 & t \\ 0 & 1 \end{pmatrix} \cdot x \cdot i \begin{pmatrix} 1 & -t \\ 0 & 1 \end{pmatrix}
\end{split}
\]
and
\[
\begin{split}
\CC \times G& \to G \\
(t,x)& \to i \begin{pmatrix} 1 & 0 \\ t & 1 \end{pmatrix} \cdot x \cdot i \begin{pmatrix} 1 & 0 \\ -t & 1 \end{pmatrix}
\end{split}
\]
respectively.

As $V$ and $W$ are complete algebraic vector fields with algebraic flows, we are interested in $Q_V$ and $Q_W$ to find functions in their respective kernels. 

We consider the following polynomial maps:
\begin{align*}
\beta_V \colon \CC \to G, & \quad
t \mapsto \phi^1_W \circ i \begin{pmatrix} 1 & t \\ 0 & 1 \end{pmatrix}
 = i \begin{pmatrix} 1-t & t \\ -t & 1+t \end{pmatrix} \\
\beta_W \colon \CC \to G, & \quad
t \mapsto i \begin{pmatrix} 1 & t \\ 0 & 1 \end{pmatrix} \\
\alpha_V \colon \CC \to Q_V, & \quad \alpha_V = \pi_V \circ \beta_V \\
\alpha_W \colon \CC \to Q_W, & \quad \alpha_W = \pi_W \circ \beta_W
\end{align*}
Since the $(\CC,+)$-actions of $V$ and $W$ arise from the action of the reductive group $\slgrp_2(\CC)$ on $G$ which is an affine variety, their ring of invariant functions is actually finitely generated, hence $Q_V$  and $Q_W$ are affine, see e.g.\ \cite{LND}*{Prop.~6.2}.

We wish to choose a sequence $\{ t_n\}_{n \in \NN} \subset \CC$ such that its image under both $\alpha_V$ and $\alpha_W$ is closed, discrete and unbounded. This is possible if $\alpha_V(\CC)$ and $\alpha_W(\CC)$ are unbounded as maps to complex-affine space.
Let us consider the differentials $\gamma_V$ and $\gamma_W$ of $\beta_V$ and $\beta_W$ at $0 \in \CC$ to obtain maps into the Lie algebra $i_\ast(\slalg_2(\CC)) \subset \mathfrak{g}$. The differentials $d\pi_V$ and $d\pi_W$ at $\id \in G$ are maps from $\mathfrak{g}$ into the tangent spaces of $Q_V$ and $Q_W$ respectively. By the third isomorphism theorem $T_{\pi_V(\id)}Q_V$ is isomorphic to the Lie algebra $\mathfrak{g}$ modulo the kernel of the projection $d_{\id} \pi_V$.
Restricting our attention to $i_\ast(\slalg_2(\CC)) \subset \mathfrak{g}$, we have that $K_V := i_\ast \slalg_2(\CC) \cap \ker d_{\id} \pi_V = \{ i_\ast \left( \begin{pmatrix} 0 & b \\ 0 & 0 \end{pmatrix}\right) : b \in \CC \}$. For $W$ we have that $K_W := i_\ast \slalg_2(\CC) \cap \ker d_{\id} \pi_W = \{ i_\ast \left( \begin{pmatrix} 0 & 0 \\ c & 0 \end{pmatrix}\right) : c \in \CC \}$. The differentials
\[
\begin{diagram}
\node{\CC} \arrow{e,t}{\gamma_V} \node{i_\ast {\slalg_2(\CC)}} \arrow{e,t}{d\pi_V} \node{{i_\ast {\slalg_2(\CC)} / K_V} \subset T_{\pi_V(\id)}Q_V} \\
\node{t} \arrow{e,t,T}{} \node{ i_\ast \left( t \cdot \begin{pmatrix} -1 & 1 \\ -1 & 1 \end{pmatrix}\right)} \arrow{e,t,T}{} \node{ \left[ t \cdot \begin{pmatrix} -1 & 1 \\ -1 & 1 \end{pmatrix}\right]} \\
\end{diagram}
\]
and
\[
\begin{diagram}
\node{\CC} \arrow{e,t}{\gamma_W} \node{i_\ast {\slalg_2(\CC)}} \arrow{e,t}{d\pi_W} \node{{{i_\ast \slalg_2(\CC)} / K_W} \subset T_{\pi_W(\id)}Q_W} \\
\node{t} \arrow{e,t,T}{} \node{ i_\ast \left( t \cdot \begin{pmatrix} 0 & 1 \\ 0 & 0 \end{pmatrix}\right)} \arrow{e,t,T}{} \node{ \left[ t \cdot \begin{pmatrix} 0 & 1 \\ 0 & 0 \end{pmatrix} \right]} 
\end{diagram}
\]
are not constant, hence $\alpha_V$ and $\alpha_W$ are not constant.

We can now make our choice of  $\{ t_n\}_{n \in \NN} \subset \CC$ such that its image under both $\alpha_V$ and $\alpha_W$ is closed, discrete and unbounded. We claim that $\left \{ i \begin{pmatrix} 1 & t_n \\ 0 & 1 \end{pmatrix} \right \}_{n \in \NN} \subset G$ is a tame sequence.

Let $\ell \colon \NN \to \NN$ denote the injective self-map we need to interpolate. As $\{\alpha_V(t_n)\}_{n\in \NN} \subset Q_V$ is closed and discrete, there exists $f \in \ker V$ such that $f\left(\varphi_W^1\left( i \begin{pmatrix} 1 & t_n \\ 0 & 1 \end{pmatrix}\right) \right)=\sqrt{\frac{t_{\ell(n)}}{t_n}}-1$. For the same reason there exists $g \in \ker W$ such that $g\left( i \begin{pmatrix} 1 & t_{\ell(n)} \\ 0 & 1 \end{pmatrix}\right) = \sqrt{\frac{t_n}{t_{\ell(n)}}}$.

Let $F=\varphi_V^f \in \aut(G)$ and $H=\varphi_W^g \in \aut (G)$, then the same computation from Lemma \ref{seq1} shows that $H^{-1} \circ F \circ \varphi_W^1$ is the interpolating automorphism.
\end{proof}

\begin{cor}
A complex semi-simple Lie group contains always a tame set.
\end{cor}
\begin{proof}
As $G$ is semi-simple, there exists an immersion of $\slgrp_2(\CC)$.
This follows from the proof of the classification of simple Lie algebras. For lack of a reference mentioning this fact, we refer to the much stronger Jacobson--Morozov Theorem \cite{basicLie}*{Theorem 3.7.2}: If $\mathfrak{g}$ is a completely reducible Lie algebra of linear transformations, then for any nilpotent element $e \in \mathfrak{g}$ there exists another nilpotent $f \in \mathfrak{g}$ and $h = [f,g]$ such that $e,f,h$ span $\slalg_2(\CC)$ as an embedded subalgebra of $\mathfrak{g}$.
Let $i$ be the induced immersion $\slgrp_2(\CC) \hookrightarrow G$.
\end{proof}

\begin{theorem}
\label{thm-lingrp}
Every connected complex-linear algebraic group $X$ different from the complex line $\CC$ or the punctured complex line $\CC^\ast$ contains an $\aut(X)$-tame set.
\end{theorem}

\begin{proof}
By Mostow's Theorem \cite{Mostow} the group $X$ is isomorphic to a semi-direct product $N \rtimes M$ where $N$ is the connected normal subgroup consisting of the unipotent elements of $X$ and $M$ is a maximal fully reductive subgroup of $X$. Hence, $X$ is isomorphic as affine variety to the direct product $N \times M$.
If both $M$ and $N$ are non-trivial complex Lie groups, we can apply Lemma \ref{lem-product} directly.
Since $N$ is a unipotent group, it is also nilpotent and hence as affine variety isomorphic to $\CC^n$. If $M$ is trivial, then by assumption we have $n \geq 2$ and we can again apply Lemma \ref{lem-product}.

It remains to consider the case where $N$ is trivial and $M$ is non-trivial. We will reduce this case to the situation where $M$ is a simple complex Lie group. 
We follow the same strategy as for the proof of the algebraic density property for complex-linear groups in \cite{KK-Criteria}*{Theorem 3}. Assume first that $M$ is indeed simple. Then we could apply Proposition \ref{prop-simplegroup} to $M$ and obtain two complete algebraic vector fields $V$ resp.\ $W$ with algebraic flows. %Our goal is to understand the behavior of these vector fields under the process of reducing $M$ to a simple Lie group.

Let $Z \cong (\CC^\ast)^n$ denote the identity component of the center of $M$ and assume it is non-trivial. If $M$ decomposes as a product $M^\prime \times Z$ we may again apply Lemma \ref{lem-product}. In the general case, $M \cong (M^\prime \times Z)/\Gamma$ for a central normal subgroup $\Gamma$. The flows of the complete vector fields will commute with $\Gamma$ and induce complete vector fields on $M$ with the same properties.

Next we may assume that the identity component of the center $Z$ is trivial. If $M$ is not simply connected, we pass to its universal cover $\widetilde M$. If a semi-simple Lie group $M$ is simply connected, then it decomposes as a product of simple Lie groups and we are done by applying Proposition \ref{prop-simplegroup} to one factor or just Lemma \ref{lem-product} in case there are at least two factors.

If $\widetilde M \to M$ is the universal cover, then $M \cong \widetilde M / H$ for a discrete (actually finite) subgroup $H$ of its center, hence as above, this induces  complete vector fields on $M$ with the same properties except that the inclusion $i \colon \slgrp_2(\CC) \to \widetilde M$ might now just be an immersion in $M$.	
\end{proof}

\begin{theorem} \label{thm-commuting}
Let $X$ be an affine algebraic complex manifold and let $V,W$ be complete algebraic vector fields whose flows are algebraic. If $[V,W]=0$ and their kernels are not contained one into the other, then there exists a tame sequence in $X$.

Moreover if $X$ has a holomorphic volume form and $V,W$ are volume preserving, then there exists a very tame sequence in $X$.
\end{theorem}

\begin{proof}
We will first provide a tame sequence, then explain how to obtain a very tame one under the additional hypothesis.

As the kernels are not contained one into the other, there exist $g,h \in \CC[X]$ such that
\[
\begin{split}
V(g)=W(h)=0 \\
V(h),W(g)\neq 0.
\end{split}
\]
Denote by $\pi_V$ and $\pi_W$ the rational maps given by Theorem \ref{Wink}.
Consider the following properties for a point $x \in X$:
\begin{enumerate}[(i)]
\item $V(h)(x)\neq 0$;
\item $W(x) \neq 0$;
\item $\pi_V^{-1}(\pi_V(x))=\{\phi_V^t(x):t\in \CC\}$;
\item $x$ is not in the singular set of $\pi_V$
\end{enumerate}
and the analogous ones
\begin{enumerate}[(a)]
\item $W(g)(x)\neq 0$;
\item $V(x) \neq 0$;
\item $\pi_W^{-1}(\pi_W(x))=\{\phi_W^t(x):t\in \CC\}$;
\item $x$ is not in the singular set of $\pi_W$
\end{enumerate}
Conditions (i),(ii),(iv),(a),(b) and (d) are clearly generically satisfied. This is true also for conditions (iii) and (c) by a theorem of Rosenlicht \cite{Rosenlicht}.

We now prove that if $\hat{x}\in X$ satisfies (i), (ii), (iii) and (iv) then the rational map
\[
\begin{split}
\CC &\to Q_V \\
t &\mapsto \pi_V(\phi_W^t(\hat{x})) 
\end{split}
\]
is not constant.
By (i), the map
\[
s \mapsto h(\phi_V^s(\hat{x}))
\]
is not constant.
Assume that $\pi_V(\phi_W^t(\hat{x})) =\pi_V(\hat{x})$ for all $t \in \CC$.
Then by (iii) we have that $\phi_W^t(\hat{x}) \in \{\phi_V^s(x):s\in \CC\}$ for all $t \in \CC$.
Since $W(\hat{x}) \neq 0$ by condition (ii), the flow of $W$ starting at $\hat{x}$ is not constant, hence there must be $\hat{t} \neq 0$ and $\hat{s} \neq 0$ such that
$\phi_W^{\hat{t}}(\hat{x})=\phi_V^{\hat{s}}(\hat{x})$ and $h(\phi_V^{\hat{s}}(\hat{x}))\neq h(\hat{x})$.
We get a contradiction because $h(\phi_W^{\hat{t}}(\hat{x}))= h(\hat{x})$, since $h \in \ker W$

If $\hat{x}$ also satisfies (a), (b), (c) and (d), we obtain that the rational images
$\{\pi_V(\phi_W^t(\hat{x})) \,:\, t \in \CC\}\subset Q_V$ and $\{\pi_W(\phi_V^t(\hat{x})) \,:\, t \in \CC\} \subset Q_W$ are unbounded.

We claim that conditions (a),(b), (c) and (d) are generically true (with respect to $t \in \CC$) for points in $\{\phi_W^t(\hat{x}) :t \in \CC\}\subset X$.
This is true for condition (a) and (b) as
\[
\begin{split}
\CC &\mapsto \CC \\
t &\mapsto W(g)(\phi_W^t(\hat{x})) 
\end{split}
\]
and 
\[
\begin{split}
\CC &\to TX \\
t &\mapsto V(\phi_W^t(\hat{x}))
\end{split}
\]
are not constant, since $\hat{x}$ satisfies (a) and (b).
Conditions (c) and (d) are always true because $\{\phi_W^s (\phi_W^t(\hat{x})):s\in \CC\}=\{\phi_W^s(\hat{x}):s\in \CC\}=\pi_W^{-1}(\pi_W(\hat{x}))$, where the last equality is condition (c) for $\hat{x}$.

We obtain our tame sequence $\{x_n\}_{n \in \NN} \subset X$ by setting $x_n=\phi_W^{t_n}(\hat{x})$, for a sequence $\{t_n\} \subset \CC$ such that
$\{ \pi_V(x_n) \} \subset Q_V$ is discrete and each $x_n$ satisfies (a), (b), (c) and (d). It exists because $\{\pi_V(\phi_W^t(\hat{x})) :t \in \CC\}\subset Q_V$ is unbounded and (a), (b), (c) and (d) are generic with respect to $t$.
It remains to prove that such a sequence is tame.

An injective map $f \colon \{x_n\} \to \{x_n\}$ is nothing but a relabeling (possibly with omissions) of the points, in the the sense that for every $n\in \NN$ there is $k_n \in \NN$ such that $f(x_n)=x_{k_n}$ and $k_n \neq k_m$ for $n \neq m$.

The quasi-affine variety $Q_V$ is equal to $\Omega \setminus A$, where $\Omega$ is an affine variety and $A$ is an algebraic subset. Since $\{ \pi_V(x_n) \} \subset Q_V \subset \Omega$ is unbounded and discrete in $Q_V$ we may assume that, after passing to a subsequence, it is also discrete in $\Omega$. Since $\Omega$ is affine, there exist $\tilde{f}_1,\tilde{f}_3 \in \Olo(\Omega) \subset \Olo(Q_V)$ such that
\[
\begin{split}
\tilde{f}_1(\pi_V(x_n)) &= a_n; \\
\tilde{f}_3(\pi_V(x_{k_n})) &= a_n
\end{split}
\]
for any sequence $\{a_n\}\subset \CC$. We postpone the choice of this sequence.

We claim that $f_1:=\pi_V^* \tilde{f}_1, \  f_3:=\pi_V^* \tilde{f}_3 \in \Olo(X)$ are well-defined and in the kernel of $V$. Since $\Omega$ is affine there exist sequences of regular maps on $Q_V$ converging to $\tilde{f}_1$ and $\tilde{f}_3$ respectively. We define the pullbacks of $\tilde{f}_1$ and $\tilde{f}_3$ as the limit of the pullbacks of the respective sequences. They are in the kernel of $V$ because they are obtained as limits of invariant maps.

This gives complete vector fields $f_1V, f_3V$ which flows are determined by
\[
\begin{split}
\CC \times X& \to X \\
(t,x)& \mapsto \phi_V^{f_i(x)t} (x)
\end{split}
\]
for $i=1,3$. The time-$1$ maps are automorphisms $F_1, F_3$ of $X$ such that
\[
\begin{split}
F_1(x_n) &=\phi_V^{a_n}(x_n)=\phi_V^{a_n}(\phi_W^{t_n}(\hat{x})); \\
F_3(x_{k_n})& =\phi_V^{a_n}(x_{k_n})=\phi_V^{a_n}(\phi_W^{t_{k_n}}(\hat{x})).
\end{split}
\]
As $[V,W]=0$ the respective flows commute, hence we are looking for $f_2 \in \ker W$ such that $f_2(\phi_V^{a_n}(x_n))=t_{k_n}-t_n$. We could then set $F_2$ to be the time-1 flow of $f_2W$ and obtain the interpolating automorphism $F:=(F_3)^{-1} \circ F_2 \circ F_1$. To obtain $f_2$ we turn our attention to $Q_W$, in particular to the sequence $\{ \pi_W (\phi_V^{a_n}(x_n)) \} \subset Q_W$. If it is discrete and without repetition then we can find $\tilde{f}_2 \in \Olo(Q_W)$ such that $f_2:=\pi_W^* \tilde{f}_2$ has the required property. Since we chose $x_n$ such that (a), (b), (c) and (d) hold, the rational map
\[
\begin{split}
\CC &\to Q_W \\
t & \mapsto \pi_W(\phi_V^t(x_n) )=\pi_W(\phi_V^t(\phi_W^{t_n}(\hat{x}))=\pi_W(\phi_V^t(\hat{x}))
\end{split}
\]
is not constant, hence unbounded. We conclude the proof by choosing the sequence $\{a_n\}$ in such a way that $\{ \pi_W (\phi_V^{a_n}(x_n)) \} \subset Q_W$ is discrete and without repetition.

For the volume preserving case, recall that if we have a holomorphic volume form on $X$, a complete vector field $V$ preserving such a form and $f \in \Olo(X)$ such that $V(f)=0$, then $fV$ is also complete and volume preserving. Hence the automorphisms $F_1,F_2$ and $F_3$ are volume preserving if so are $V$ and $W$.
\end{proof}

Observe that even if we start from algebraic vector fields with algebraic flows, it is not true in general that the interpolating automorphism will be algebraic.

\begin{cor}
Let $G$ and $H$ be non-trivial affine complex Lie groups whose connected components are not biholomorphic to $(\CC^\ast)^n$. Then $G \times H$ contains a $\aut(G \times H)$-tame sequence.
\end{cor}

\begin{cor}
\label{cor-KorasRussell}
Each variety of the family $X_{a,b}=\{x^2 y= a(z)+xb(z)\} \subset \CC^{n+3}$ for $a,b \in \Olo(\CC^{n+1})$ contains a very tame sequence with respect to the holomorphic volume form $\omega=\frac{dx}{x^2}\wedge dz_0 \wedge \dots \wedge dz_n$.
In particular, the \textit{Koras-Russell cubic threefold} $C=\{x^2y+x+z^2+w^3=0\} \subset \CC^4$ contains a very tame sequence with respect to the holomorphic volume form $\omega=\frac{dx}{x^2}\wedge dz \wedge dw$.
\end{cor}

This class of complex varieties was recently considered by Leuenberger \cite{Leuenberger}. He proved that under suitable conditions on the holomorphic functions $a$ and $b$, the space $X_{a,b}$ is a complex manifold with the density property and in some cases it also has the volume density property. Here we include the case where the variety is singular, as we will use vector fields that vanish on the singular locus.

The manifold $C$ is well-known to be an affine algebraic manifold which is diffeomorphic to $\RR^6$ but not algebraically equivalent to $\CC^3$ \cite{MakarLimanov}. The fact that it has the density and volume density property rises the question of whether $C$ is holomorphically equivalent to $\CC^3$. Related to the issue there is a conjecture stating that a manifold with the density property diffeomorphic to $\CC^n$ should be biholomorphic to $\CC^n$ \cite{TothVarolin2006}.

\begin{proof}
The vector fields
\[
V = \left( \frac{\partial a}{\partial z_0} + x \frac{\partial b}{\partial z_0} \right) \frac{\partial}{\partial y} + x^2 \frac{\partial}{\partial z_0}
\]
and
\[
W = \left( \frac{\partial a}{\partial z_1} + x \frac{\partial b}{\partial z_1} \right) \frac{\partial}{\partial y} + x^2 \frac{\partial}{\partial z_1}
\]
satisfy the hypothesis of Thm \ref{thm-commuting}.
\end{proof}

\section{Unavoidable sets}
\label{sec-unavoidable}

In analogy to quasi-affine varieties introduce the following notion:
\begin{definition}
A complex manifold $X$ is called \emph{quasi-Stein} if it biholomorphic to an open subset of a Stein manifold.
\end{definition}

\begin{theorem}
\label{thm-unavoid}
Let $V$ be an complex-affine algebraic manifold, $X$ a quasi-Stein manifold and assume that $\dim V = \dim X \geq 1$. Then there exists an infinite discrete closed subset $D \subset X$ such that $F(V) \cap D \neq \emptyset$ for every non-degenerate holomorphic map $F \colon V \to X$.
\end{theorem}

This Theorem has already been established by Winkelmann \cite{Winkelmann01} for $V$ being an irreducible affine-algebraic variety and $X$ being weakly Stein.
We give an entirely different proof of this theorem following more closely the ideas of Rosay and Rudin \cite{RosayRudin} where $V = X = \CC^n$. In the final step of the proof we will apply the Ohsawa--Takegoshi $L^2$-Extension Theorem. The algebraicity of $V$ will only be needed at this point, in combination with the $L^2$- estimate.

\medskip

Let $X$ be a complex manifold with an exhaustion function $\eta_X \colon X \to [0, +\infty)$. For $\rho > 0$ we denote the sub-level sets of $\eta_X$ by
\[
X_\rho := \{ z \in X \,:\, \eta(z) < \rho \}
.\]

The following lemma is due to Rosay and Rudin \cite{RosayRudin}*{Lemma~4.3} for the special where $V$ and $X$ are complex-Euclidean spaces exhausted by balls. Their proof carries over to the general case.
\begin{lemma}
\label{lem-exhaust1}
Let $V$ and $X$ be complex varieties, $X$ quasi-Stein, with fixed exhaustion functions and fixed Riemannian metrics. Let $v_0 \in V$ and $x_0 \in X$ be in the $0$-level set of their respective exhaustion functions. Given $0 < a < b$ and $0 < r < s$ and $c \geq 0$, let $\Gamma$ be the class of all holomorphic mappings $F \colon V_{b} \to X_{s}$ such that $F(v_0) \in X_{r/2}$ and 
$\displaystyle \max_{\overline{V_{a}} \cap V_{\mathrm{reg}}} | d F | \geq c$.
Then there exists a finite set $E = E(a,b,r,s,c) \subset \boundary X_{r}$ with the following property:
\begin{equation}
\label{eqhitfinite}
F \in \Gamma \text{ and } F(V_{a}) \cap \boundary X_{r} \neq \emptyset \Longrightarrow F(V_{b}) \cap E \neq \emptyset
\end{equation}
\end{lemma}
\begin{proof}
Without loss of generality we may assume that $V$ is connected and $X \subseteq \CC^N$ is a not necessarily closed subvariety.
Let $E_1 \subset E_2 \subset E_3 \subset \dots$ be increasing finite subsets of $\boundary X_{r}$ each of which fails to satisfy \eqref{eqhitfinite} in place of $E$, but whose union is dense in $\boundary X_{r}$. Hence, for each $j \in \NN$ there exist $F_j \in \Gamma$ and $z_j \in V_{a}$ such that $F(z_j) \in \boundary X_{r}$ and $F_j(V_{b}) \cap E_j =	 \emptyset$.

Since $\Gamma$ is a normal family and $\overline{V_{a}}$ is compact, we may, after passing to a subsequence, assume that $z_j \to z_0 \in \boundary V_{a}$ and $F_j \to F \in \Gamma$, uniformly on compacts of $X_{b}$.

We then have
\[
F(z_0) = \lim_{j \to \infty} F_j(z_j) \in \boundary X_r
\]
and since $F \in \Gamma$ we have a lower bound for the derivatives which implies that
\[
\Omega := \{ z \in V_b \,:\, \rank d_z F = \dim V \}
\]
is not empty. Hence $\Omega$ is a connected open set that is dense in $V_b$ (since the complement is of complex codimension at least one) and because $F$ is a continuous and open map, the set $F(\Omega)$ is connected, open and dense in $F(V_b)$.

Since $F$ is an open map with $F(z_0) \in \boundary X_r$ and $z_0 \in V_b$, the set $F(V_b)$ contains points outside $\overline{V_b}$, hence so does $F(\Omega)$. Due to $F(v_0) \in X_{r/2}$, the set $F(\Omega)$ must intersect $X_r$, and -- being connected -- also $\boundary X_r$.

We can now choose a point $p \in \Omega$ with $F(p) = q \in \boundary X_a$. Since $d_p F$ has rank $\dim V$, we find a compact neighborhood $K$ of $p$ inside local coordinates such that $F|K$ is a one-to-one map. For any relatively compact open subset $L \subset\subset F(K)$ there is an index $j_0$ such that $j \geq j_0$ implies that $F_j(K) \supseteq L$, hence $L \subseteq F_j(X_b)$. But this leads to a contradiction for large enough $j$, since $L$ contains points of $E_j$.
\end{proof}

\begin{lemma}
\label{lem-exhaust2}
Let $V$ and $X$ be complex varieties, X quasi-Stein, with fixed exhaustion functions and fixed Riemannian metrics. Let $v_0 \in V$ and $x_0 \in X$ be in the $0$-level set of their respective exhaustion functions.
Then for each $t > 0$ there exists a discrete set $E_t \subset X \setminus X_t$ such that the following assumptions
imply that $F(V_{t/2}) \subset X_t$:
\begin{enumerate}
\item $F \colon V_t \to X$ is holomorphic
\item $F(v_0) \in X_{t/2}$
\item $\max\limits_{\overline{V_{t/2}} \cap V_{\mathrm{reg}}} |d_z F| \geq 1/t$
\item $F(V_t) \cap E_t = \emptyset$
\end{enumerate}
\end{lemma}

\begin{proof}
Choose sequences $a_j \in \RR$ and $r_j \in \RR$ such that
\[ \frac{1}{2} t = a_1 < a_2 < \dots < \frac{3}{4}t \text{ and } t = r_1 < r_2 < \dots \]
and $\lim_{j \to \infty} r_j = \infty$.
Using the preceding Lemma \ref{lem-exhaust1} and its notation, define
\[
E_t := \bigcup_{j=1}^\infty E(a_j,a_{j+1},r_j,r_{j+1},1/t)
\]
It is easy to see that $E_t$ is closed and discrete.
Now assume that a map $F$ satisfying the properties above is given.
Clearly, $F(V_{a_{j+1}}) \subset X_{r_{j+1}}$ for some $j \in \NN$.
By assumption, $F(V_{a_{j+1}})$ does not intersect $E(a_j,a_{j+1},r_j,r_{j+1},1/t)$ and the application of the preceding Lemma shows that hence $F(V_{a_j})$ does not intersect $\boundary X_{r_j}$. This implies that $F(V_{a_j}) \subseteq X_{r_j}$.

We can now replace $X$ by $X_{r_j}$ and $V$ by $V_{a_j}$ and repeat the last step until we reach $F(V_{a_1}) \subseteq X_{r_1}$.
\end{proof}

We are now ready for the proof of the Theorem.

\begin{proof}[Proof of Theorem \ref{thm-unavoid}]
We apply Lemma \ref{lem-exhaust2} and define the set $D := \bigcup_{t \in \NN} E_t$. Clearly, $D$ is closed and discrete. For large enough $t$, the non-degenerate holomorphic map $F \colon V \to X$ must satisfy $F(V_{t/2}) \subseteq X_t$. Since $X$ is a (not necessarily closed) complex subvariety of some complex-Euclidean space, we can consider $F$ as vector of functions on $V$. Because $V \subseteq \CC^n$ is a complex subvariety of some $\CC^n$, we can extend $F$ from $V \cap r \udisc^n$ holomorphically to $\hat{F}_r$ on $r \udisc^n$ with $L^2$-estimates. Using Cauchy-estimates for the derivatives, we obtain:
\begin{align*}
\left| \frac{\partial^{|\alpha|} \hat F_r}{\partial z^\alpha}(0) \right| &\leq \frac{1}{(2\pi)^n} \frac{(a+2)!}{(\alpha + 1)^{(1,\dots,1)}} \frac{1}{r^{|\alpha| + 2n}} \int_{(r\udisc)^n} |\hat F_r| d\mathcal{L}^n \\
&\leq \frac{1}{(2\sqrt{\pi})^n} \frac{(a+2)!}{(\alpha + 1)^{(1,\dots,1)}} \frac{1}{r^{|\alpha| + n}} \left\| \hat F_r \right\|_{L^2(r\udisc^n)} \\
&\leq \frac{1}{(2\sqrt{\pi})^n} \frac{(a+2)!}{(\alpha + 1)^{(1,\dots,1)}} \frac{1}{r^{|\alpha| + n}} \cdot C_r(V) \cdot \left\| F \right\|_{L^2(r\udisc^n \cap V)} \\
&\leq \frac{1}{(2\sqrt{\pi})^n} \frac{(a+2)!}{(\alpha + 1)^{(1,\dots,1)}} \frac{1}{r^{|\alpha|}} \cdot C_r(V) \cdot 2^n
\end{align*}
where the last inequality holds for $r > 0$ large enough and if we assume that the exhaustion functions on $V$ and $X$ were chosen as restriction of the Euclidean norm of the ambient space.
The constant $C_r(V)$ is provided by the Ohsawa--Takegoshi $L^2$-Extension theorem and depends on the defining equations and grows polynomially in $r > 0$, since the defining equations of $V$ are polynomial. For a multi-index $\alpha \in \NN_0^n$, we obtain in the limit $r \to \infty$ that \[
 \frac{\partial^{|\alpha|} \hat F_r}{\partial z^\alpha}(0) \to 0 \qquad \text{ for } |\alpha| \geq \deg_r C_r(V) + 1,
\]
hence the higher derivatives of $F$ restricted to the tangent directions of $V$ must vanish and so $F$ is a polynomial map of degree at most $\deg_r C_r(V)$.

The growth rate does not depend on the defining equations of $X$, hence we may replace $X \subseteq \CC^m$ by the graph of an exponential function over $X$, i.e.\ $\{ (x,y) \in \CC^{m} \times \CC \,:\, y = \exp(x) \}$ for the whole construction. This excludes the existence of any such non-trivial polynomial maps and hence concludes the proof.
\end{proof}

\section{Open Questions}

\begin{question}
Let $X$ be a Stein manifold with the density property. Is a weakly tame set always (strongly) tame?
\end{question}

The following question has been asked for $X = \CC^n$ already by Rosay and Rudin \cite{RosayRudin}.

\begin{question}
Let $X$ be a Stein manifold  with the density property. Let $Z \subset X$ be a subset such that every bijection $Z \to Z$ is the restriction of a holomorphic automorphism of $X$. Is $Z$ tame?
\end{question}

Next, consider the following observation:
\begin{remark}
Let $X$ be a complex manifold that contains a tame set. Then $\aut(X)$ is infinite dimensional.
\end{remark}
\begin{proof}
Let $\dim X = n$. If a group $G$ acts transitively on $m$-tupels of points from $X$,
then $\dim G \geq m \cdot n$, since $G$ acts transitively on $X^m \setminus \mathrm{diag}(X^m)$, i.e.\ the configuration space of $m$-tupels of points from $X$, which has dimension $m \cdot n$.
\end{proof}

Typical examples of spaces with an infinite dimensional automorphism groups are the so-called flexible manifolds which comprise also the Stein manifolds with the density property.

\begin{question}
Let $X$ be a Stein manifold that contains a tame set and such that $\aut(X)$ acts transitively on $X$. Is $X$ holomorphically flexible? Does $X$ enjoy the density property?
\end{question}

Note that ``tame'' is not an algebraic condition, but nevertheless the same question can be asked also in a semi-algebraic context:

\begin{question}
Let $X$ be an affine-algebraic manifold that contains a tame set and such that $\aut_{\mathrm{alg}}(X)$ acts transitively on $X$. Is $X$ algebraically flexible? Does $X$ enjoy the (algebraic) density property?
\end{question}

In Stein manifolds with the density property there is at most one tame set up to equivalence. It is also easy to provide examples that admit no tame sets, yet the following question remains open:

\begin{question}
Is there a complex manifold admitting non-equivalent tame sets?
\end{question}

\begin{question}
Let $X$ be a Stein manifold with the density property and let $Z \subset X$ be a tame set. Is $X \setminus Z$ an Oka manifold?
\end{question}

It is well-known that $\CC^n \setminus (\NN \times \{0\}^{n-1})$ and hence the complement of any tame set in $\CC^n$, is an Oka manifold for $n \geq 2$, see \cite{FrancBook}*{Prop.~5.6.17}. We show that the question has an affirmative answer at least for another example.

\begin{example}
The manifold $\slgrp_2(\CC) \setminus Z$ where $Z$ is a tame set, is Oka.
Since $\slgrp_2(\CC)$ has the density property, it is suffient to show this for one particular tame set. We choose
\[
Z := \left\{ \begin{pmatrix} 1 - k & k \\ -k & 1 + k \end{pmatrix} \;:\; k \in \NN \; \text{prime} \right\}
\]
which is tame according to the proof of Lemma \ref{seq1}.
We consider the the following vector fields that arise from left-multiplication and hence are nowhere vanishing and complete, see also the proof of Lemma \ref{seq2}. Moreover, they span the tangent space of $\slgrp_2(\CC)$ in each point.
\begin{align*}
A &= c \frac{\partial}{\partial a} + d \frac{\partial}{\partial b} \\
B &= a \frac{\partial}{\partial c} + b \frac{\partial}{\partial d} \\
[A,B] &= - a \frac{\partial}{\partial a} - b \frac{\partial}{\partial b} + c \frac{\partial}{\partial c} - d \frac{\partial}{\partial d} 
\end{align*}
Also note that $c, d \in \ker A$, $a, b \in \ker B$ and $ac, ad, bc, bd \in \ker[A,B]$.
By the Mittag-Leffler Theorem we find holomorphic functions $\alpha(a)$, $\beta(b)$, $\gamma(c)$, $\delta(d)$ such that the complete vector fields $\gamma A, \delta A, \alpha B, \beta B$ vanish exactly in
\[
\widetilde{Z} := \left\{ \begin{pmatrix} 1 - k_1 & k_2 \\ -k_3 & 1 + k_4 \end{pmatrix} ;\:\; k_1, \dots, k_4 \in \NN \; \text{prime} \right\}
\]
Next we arrange for those integers $k_1, \dots, k_4$ to be equal: choose a holomorphic function $\varepsilon(-bc)$ such that it vanishes only for squares of primes. Then $\varepsilon(-bc) \cdot [A,B]|\widetilde Z$ will vanish only when $k_2 = k_3 =: k$ . Similarly we can also find a holomorphic function $\zeta(bd)$ to force $k_4 = k$ using $\zeta(bd) \cdot [A, B]$. In $\slgrp_2(\CC)$ this implies $k_1 = k$ as well.
We have found six complete vector fields on $\slgrp_2(\CC)$ whose common zero locus is exactly $Z$. 

\smallskip

In each point where all coordinates are non-integers, these vector fields will span the tangent space. Let $p \in \slgrp_2(\CC) \setminus Z$ be a point where one coordinate is an integer. Then at least one of the above-mentioned vector fields is non-vanishing. Following its and perhaps others' flow for an arbitrarily short time, we will eventually reach a point $q$ with all non-integer coordinates. Hence, the pull-backs by this flow of the vector fields spanning in $q$ will span in $p$. An infinitesimal pull-back is a Lie derivative, hence we just form all possible Lie brackets ($15$) of these six vector fields to be on the safe side.

Hence $\slgrp_2(\CC) \setminus Z$ is elliptic in the sense of Gromov and therefore an Oka manifold. Moreover, it is holomorphically flexible in the sense of Arzhantsev~et~al.
\end{example}

\end{document}